\documentclass[10pt,reqno]{amsart}
\usepackage{amsmath}
\usepackage{amssymb}
\usepackage{amsthm}
\usepackage{verbatim}
\usepackage{graphicx}
\usepackage{marvosym}
\usepackage[table]{xcolor}
\usepackage{soul}
\usepackage{amsrefs}
\usepackage{enumitem}
\usepackage{eufrak}

\newtheorem{theorem}{Theorem}[section]

\newtheorem{corollary}[theorem]{Corollary}
\newtheorem{proposition}[theorem]{Proposition}

\theoremstyle{remark}
\newtheorem{definition}[theorem]{Definition}
\newtheorem{example}{Example}
\newtheorem*{theo}{Theorem}
\newtheorem*{question}{Question}

\newtheorem{remark}[theorem]{Remark}

\newtheorem*{examples}{Examples}

\newcommand{\spans}{\operatorname{span}}
\newcommand{\HH}{\mathcal{H}}
\newcommand{\BH}{\mathcal{B}(\mathcal{H})}
\newcommand{\KH}{\mathcal{K}(\mathcal{H})}
\newcommand{\FH}{\mathcal{F}(\mathcal{H})}
\newcommand{\I}{\mathcal{I}}
\newcommand{\LL}{\mathcal{L}}

\DeclareMathOperator{\diag}{diag}

\begin{document}
\title[On Simplicity of Lie algebras]{On Simplicity of Lie algebras of Compact operators: \\A Direct Approach}

\author{Sasmita Patnaik}
\address{S.~Patnaik, Department of Mathematics and Statistics, Indian Institute of Technology, Kanpur, Uttar Pradesh 208016 (INDIA)}
\email{sasmita@iitk.ac.in}


\maketitle
\begin{abstract}
We investigate an algebraic variant of the Wojty$\acute{\text{n}}$ski problem on the simplicity of Lie algebras of compact operators on a separable infinite-dimensional complex Hilbert space. We prove the non-simplicity of Lie algebras of compact operators under a mild softness condition using the notion of soft-edged operator ideals introduced by Kaftal and Weiss. We believe our study offers a new perspective on the investigation of the simplicity of Lie algebras by relating it to the study of operator ideals. 
\end{abstract}

\section{Introduction}
A long-standing open question that has called the attention of many mathematicians is the 
Wojty$\acute{\text{n}}$ski problem \cite[Question 3]{W76}: Does every closed Lie algebra of compact quasinilpotent operators on a Banach space contain a nontrivial closed Lie ideal? Several partial results centered around this question have been obtained by Bre$\check{\text{s}}$ar, Shulman, and Turovskii \cite{BST10}. Historically, to the best of our knowledge, consideration of the adjoint representation of elements of Lie algebras in the study of the simplicity of Lie algebras seemed inevitable. The purpose of this paper is to offer a new (and more direct) approach that bypasses the use of adjoint representation in the study of the simplicity of Lie algebras of operators that act on a Hilbert space. We achieve this by implementing advanced operator ideal techniques and establish a bridge connecting a purely algebraic problem with an analytic behavior of certain operator ideals. We begin our study by focussing on the algebraic simplicity of Lie algebras. (That is, simplicity of Lie algebras without any topology.)

The theme of this paper is to address the question: \textit{Which infinite-dimensional Lie algebras of compact operators on a Hilbert space are (algebraically) simple?}
We then obtain a partial answer to the following natural algebraic analog of the above Wojty$\acute{\text{n}}$ski problem in the Hilbert space setting: Does every Lie algebra of compact quasinilpotent operators on a Hilbert space contain a nontrivial Lie ideal?  
In this paper, we investigate questions on the existence of nontrivial Lie ideals (not necessarily closed in any topology) of Lie algebras of compact operators via the notion of soft-edged ideals in the theory of operator ideals. Soft-edged ideals (soft ideals, in short) were first introduced by Kaftal and Weiss \cite[Definition 6]{KW02}. Soft ideals have played a significant role in various aspects of operator ideal theory, see \cite{D66}, \cite{KW02}, \cite{KW07}, \cite{M64}, \cite{P80}-\cite{PW13} for instance. Our main result, Theorem \ref{theorem} provides a criterion involving soft ideal condition that guarantees the (algebraic) non-simplicity of Lie algebras of compact operators. In the end, we discuss the simplicity of Lie algebras of finite rank operators in the Hilbert space framework.

\section{Preliminaries}
Let $\HH$ denote a separable infinite-dimensional complex Hilbert space and $\BH$ the algebra of bounded operators acting on $\HH$.
By an operator ideal, we mean a two-sided ideal $\I$ of $\BH$ (henceforth called $\BH$-ideal). We say that a $\BH$-ideal $\I$ is nontrivial, if $\{0\} \neq \I \neq \BH$. We first provide the necessary background required to understand the advanced operator ideal techniques that we employ in addressing the aforementioned questions on the simplicity of Lie algebras.

\subsection*{Background} \label{background}
Calkin established the inclusions $\FH \subset \I \subset \KH$ for a nontrivial $\BH$-ideal $\I$  \cite[Theorem 1.7]{C41}, where $\FH$ is the $\BH$-ideal of finite rank operators and $\KH$ is the $\BH$-ideal of compact operators. For a detailed background on $\BH$-ideals, one can refer to the classic book by Gohberg and Krein \cite{GK1969}. 

The $\BH$-ideals were completely characterized by Calkin \cite{C41} who showed that there is a one-to-one correspondence between the proper $\BH$-ideals and the so-called ideal sets (or characteristic sets), denoted by $\Sigma$, of nonnegative numbers decreasing to zero, i.e., $\Sigma \subset c^*_0$. These ideal sets are hereditary (i.e., solid under positive order) cones that are invariant under ampliations. For each $m \in \mathbb N$, the $m$-fold ampliation $D_m$ is defined by:
$$c^*_0 \ni \xi \rightarrow D_m\xi := \left<\xi_1, \cdots, \xi_1, \xi_2, \cdots, \xi_2, \cdots\right>$$
with each entry $\xi_i$ of $\xi = \left<\xi_n\right>$ repeated $m$ times. For a given ideal $\I$, the characteristic set associated with $\I$ is given by $\Sigma(\I) := \{\eta \in c^*_0 \mid \diag \eta \in \I\}$. 
The Calkin correspondence $\I \rightarrow \Sigma(\I)$ induced by $$\I \ni T \rightarrow s(T) \in \Sigma(\I),$$ where $s(T)$ denotes the sequence of singular numbers of $T$, is a lattice isomorphism between $\BH$-ideals and characteristic sets and its inverse is the map from a characteristic set $\Sigma$ to the ideal generated by the collection of the diagonal operators $\{\diag \xi \mid \xi \in \Sigma\}$.
Calkin's characterization of $\BH$-ideals essentially reduces the entire theory of (non-commutative) $\BH$-ideals to the much simpler theory of (commutative) sequence spaces. The insight into problems concerning $\BH$-ideals gained by this approach often enables one to solve problems related to $\BH$-ideals.

We next give a brief description of certain properties of $\BH$-ideals via the Calkin's characterization that are relevant for this paper.  

\begin{remark}\label{facts}{Some standard facts on $\BH$-ideals:}
\begin{enumerate}[label=(\roman*)]
\item \cite[Section 4]{KW07} If $\I$ and $\mathcal J$ are $\BH$-ideals, then the product $\I\mathcal J$ is again a $\BH$-ideal and its characteristic set is given by $$\Sigma(\I\mathcal J) = \{\xi \in c^*_0 \mid \xi \leq \eta\rho \text{ for some } \eta \in \Sigma(\I) \text{ and } \rho \in \Sigma(\mathcal J)\}.$$

\item \cite[Section 1]{KW07} For $T \in \KH$, let $(T)$ denote the principal ideal of $\BH$ generated by $T$. Then, $$A \in (T) \text{ if and only if } s(A) = O(D_m(s(T))) \text{ for some } m \geq 1.$$  

\noindent (For $\xi, \eta \in c^*_0$, $\xi = O(\eta)$ means that $\xi_n/\eta_n \leq M$ for some constant $M>0$ and for all $n\geq 1$.) \\

\item \cite[Section 2, first paragraph]{KW07} For a $\BH$-ideal $\I$, $A \in \I$ if and only if $A^* \in \I$ if and only if $|A| \in \I$ (via the polar decomposition $A = V|A|$ and $V^*A = |A| := (A^*A)^{1/2}$) if and only if $\diag s(A) \in \I$. Consequently, $(A) = (|A|) = (\diag s(A))$.

\item When $T = \displaystyle{\sum^{n}_{i=1}}A_iTB_i$ with each $A_i$ or $B_i$ in $\KH$, $s(T) = o(D_m(s(T)))$ for some $m >1$. Indeed, $T \in (T)\KH$ and so using the above two facts (i)-(ii), one obtains $s(T) = O(D_m(s(T))s(C))$ for some $m>1$ and $C \in \KH$. And then it follows from $\Sigma(\KH) = c^*_0$ that $s(T) = o(D_m(s(T)))$.  
(For $\xi, \eta \in c^*_0$, $\xi = o(\eta)$ means that $\xi_n/\eta_n \rightarrow 0$ as $n \rightarrow \infty$.)
\end{enumerate}
\end{remark}

\subsection*{Soft ideals in $\BH$}\label{soft}
\begin{definition}\label{soft}\cite[Definition 6]{KW02} An ideal $\I$ of $\BH$ is called a soft-edged ideal (i.e., soft ideal, in short) if $$\I = \I \KH.$$

\noindent In other words, $\I$ is soft-edged if and only if for each $\xi \in \Sigma(\I)$, one has $\xi = o(\eta)$ for some $\eta \in \Sigma(\I)$, see \cite[Section 4, first paragraph following Definition 4.1]{KW07}. 

\end{definition}

\begin{examples}\label{E1} Some examples of soft ideals and non-soft ideals of $\BH$ are listed below.
\begin{enumerate}[label=(\roman*)]
\item The Gohberg and Krein ideals $\mathfrak{S}_{\phi}$ generated by symmetric norming functions $\phi$ are soft ideals if and only if $\phi$ is mononormalising \cite[Chapter III, Sections 6-7]{GK1969}. In particular, all Schatten $p$-class ideals ($p\geq1$) are soft ideals [ibid, Section 7].

\item An idempotent $\BH$-ideal $\I$ is a soft ideal. Indeed, $$\I = \I^2 \subset \I\KH \subset \I.$$
Hence, one obtains equality, i.e., $\I = \I\KH$. For instance, $\FH$ and $\KH$ are idempotent ideals and hence soft ideals.

\item The principal $\BH$-ideal generated by the positive diagonal operator $\diag\left<1/n\right>$ is not a soft ideal. If it were, then by Definition \ref{soft}, $(\diag\left<1/n\right>) = (\diag\left<1/n\right>)\KH$. This implies that $\left<1/n\right> \in \Sigma((\diag\left<1/n\right>)\KH)$. By Remark \ref{facts}(i)-(ii) and that $\Sigma(\KH) = c_0^*$, one obtains $\left<1/n\right> = o (D_m\left<1/n\right>)$ for some $m > 1$. The $k^{th}$ entry of the sequence $\left<\frac{1/n}{D_m\left<1/n\right>}\right>$ where $k = mj+r$ given by $\frac{1/(mj+r)}{1/(j+1)} \rightarrow 1/m$ as $k \rightarrow \infty$, contradicting $\left<1/n\right> = o (D_m\left<1/n\right>)$.
\end{enumerate}
\end{examples}

The following theorem of Fong and Radjavi, \cite[Theorem 1:(v)-(vi)]{FR83}, ties together an algebraic condition with an analytic condition related to softness of ideals. We mention it below for completeness.
We point out that in the case of principal $\BH$-ideals, one derives an efficient criterion to test the softness of such ideals, see Corollary \ref{C1} below. 

\begin{theo}\cite[Theorem 1]{FR83} Let $P$ be a positive compact operator of infinite rank. The following are equivalent.
\begin{enumerate}[label=(\roman*)]
\item $P = A_1PB_1 + \cdots + A_kPB_k$ for some $k >1$ where $A_i, B_i \in \BH$ and either the $A_i$ or the $B_i$ are compact operators.  
\item For some integer $k>1$, $s_{nk}(P) = o(s_n(P))$ as $n \rightarrow \infty$.
\end{enumerate}
(The $k$ in (i) and (ii) need not be the same, see \cite[Remarks (4)-(5)]{FR83} for the constraints on $k$.)
\end{theo}

As a consequence, one immediately obtains

\begin{corollary}\label{C1}
Let $P$ be a positive compact operator of infinite rank. The following are equivalent.
\begin{enumerate}[label=(\roman*)]
\item The principal ideal $(P)$ is a soft ideal.
\item For some integer $k>1$, $s_{nk}(P) = o(s_n(P))$ as $n \rightarrow \infty$. 
\end{enumerate}
\end{corollary}

\begin{remark}\label{s-number}
(i) If the singular number sequence of the operator $P$ is given by $s(P) = \left<1/2^n\right>$, then $(P)$ is a soft ideal; but if $s(P) = \left<1/n\right>$, then $(P)$ is not a soft ideal. Indeed, $s_{nk}(P)/s_n(P)=2^n/2^{nk} = 1/2^{n(k-1)} \rightarrow 0$, but $s_{nk}(P)/s_n(P) = n/nk = 1/k \not \rightarrow 0$ as $n \rightarrow \infty$.

(ii) For $(T)$ a principal ideal, if $(T)$ is a soft ideal, then by Remark \ref{facts}(iv), a necessary condition is that $s(T) = o(D_m(s(T)))$ for some $m >1$.

(iii) A sequence $\xi := \left<\xi_n\right> \in c^*_0$ satisfies the $\Delta_2$-condition if $\sup \{\xi_n/\xi_{2n}\} < \infty$ \cite[second paragraph following Definition 6]{KW02}. A sufficient condition that guarantees the non-softness of a principal ideal $(T)$ is that if $s(T)$ satisfies the $\Delta_2$-condition, then $(T)$ is not a soft ideal. Indeed, if $\sup \{s_n(T)/s_{2n}(T)\} < \infty$, then it follows that $D_2(s(T)) \leq Cs(T)$ for some constant $C>0$ which further implies that $D_k(s(T)) \leq Cs(T)$ for each $k>2$. That is, for each $k\geq 2$, the sequence $s(T)/D_k(s(T)) \geq 1/C \not \rightarrow 0$ contradicting the above necessary condition (ii) for the softness of $(T)$.  
\end{remark}

\subsection*{Lie algebras of compact operators}\label{Lie}
\begin{definition}\label{Lie algebra}
A Lie algebra of operators $\LL$ is a subspace of $\BH$ which is closed under the commutator product $[A, B] = AB - BA$ for $A, B \in \LL$.

A Lie ideal $\mathcal{J}$ of $\LL$ is a Lie subalgebra of $\LL$ such that $[A, B] \in \mathcal J$ whenever $A \in \LL$ and $B \in \mathcal J$.
\end{definition}

A Lie algebra $\LL$ is called a \textit{simple} Lie algebra (or algebraically simple Lie algebra) if there is no nontrivial Lie ideal of $\LL$. (The trivial Lie ideals are $\{0\}$ and $\LL$.)
In this paper, we consider (algebraic) Lie ideals, i.e., Lie ideals without any topology. 
 
For a Lie algebra $\LL$, let $[\LL, \LL]$ denote the (linear) span of the set $\{[A, B] \mid A, B \in \LL\}$. It is straightforward to see that $[\LL, \LL]$ is a Lie ideal of $\LL$ which is often called the derived algebra of $\LL$.

\section{Simplicity of Lie algebras of compact operators: a soft criteria}
Throughout this section, $\LL$ is an infinite-dimensional Lie algebra of compact operators acting on a separable infinite-dimensional complex Hilbert space $\HH$ and nonabelian (i.e, $[\LL, \LL] \neq \{0\}$).

A natural class of infinite-dimensional Lie algebras of compact operators that are nonabelian is the nonzero $\BH$-ideals which inherit naturally the Lie algebra structure via the commutator operation. We begin with the following proposition for this class of Lie algebras. 

\begin{proposition}\label{ideal}
Let $\LL$ be a Lie algebra of compact operators. If $\LL$ is a nonzero $\BH$-ideal, then it is automatically not simple.
\end{proposition}
\begin{proof}
If $\LL$ is a nonzero $\BH$-ideal that contains an infinite rank operator, then $\LL$ is not simple because $\FH  \subsetneq \LL$ is a nontrivial Lie ideal of $\LL$ (\cite[Chapter III, Section 1, Theorem 1.1]{GK1969}). If $\LL = \FH$, then the derived algebra $[\FH, \FH]$ is a nontrivial Lie ideal of $\LL$ because $[\FH, \FH] \neq 0$ and $[\FH, \FH]$ consists of trace-zero finite rank operators which is clearly a proper subset of $\FH$. 
\end{proof}

We next consider Lie algebras that are not $\BH$-ideals and investigate its simplicity. 
We denote by $(\LL)$ the $\BH$-ideal generated by a Lie algebra $\LL$.
We begin with a qualitative analysis of the derived ideal $[\LL, \LL]$ of $\LL$.
\begin{proposition}\label{derived-ideal}
For $\LL$ a Lie algebra of compact operators, if $[\LL, \LL] = \LL$, then $(\LL)$ is idempotent.
\end{proposition}
\begin{proof}
Recall that $[\LL, \LL] = \spans\{[A, B] \mid A, B \in \LL\}$. Since $A, B \in \LL$, the products $AB, BA \in (\LL)(\LL)$. Hence $[\LL, \LL] \subset (\LL)^2$. By hypothesis, $\LL = [\LL, \LL]$. Combining these observations, one has
$$\LL = [\LL, \LL] \subset (\LL)^2.$$
By Remark \ref{facts}(i), $(\LL)^2$ is a $\BH$-ideal, so $(\LL) \subset (\LL)^2$. As $(\LL)^2 \subset (\LL)$, one obtains $(\LL)^2 = (\LL)$, i.e., $(\LL)$ is idempotent.
\end{proof}
An immediate consequence is
\begin{corollary}\label{idempotent}
If $(\LL)$ is not a soft ideal, then $\LL$ is not simple.
\end{corollary}
\begin{proof}
Since $(\LL)$ is not a soft ideal, by Examples (ii) mentioned above, $(\LL)$ is not idempotent. Therefore, by Proposition \ref{derived-ideal}, $[\LL, \LL] \neq \LL$. Moreover, we have assumed that $\LL$ is nonabelian, so $[\LL, \LL] \neq \{0\}$. So $[\LL, \LL]$ is a nontrivial Lie ideal of $\LL$, thereby implying $\LL$ is not simple.
\end{proof}


\begin{remark}\label{class}
Combining Proposition \ref{derived-ideal} and Corollary \ref{idempotent}, one obtains $$[\LL, \LL] = \LL \, \Rightarrow (\LL) = (\LL)^2 \Rightarrow (\LL) \text{ is a soft ideal}.$$
Each of these implications is proper. Indeed, for the first implication, consider $\LL \subset \FH$ such that there is $F \in \LL$ with nonzero trace. Since the $\BH$-ideal generated by any subset of finite rank operators is $\FH$, so $(\LL) = \FH$. Hence, $(\LL)$ is idempotent but $\LL \neq [\LL, \LL]$ because $[\LL, \LL]\subset [\FH, \FH]$. For the second implication, consider $\LL = \mathfrak S^0_1$, the Lie algebra of all trace-zero trace class operators. The $\BH$-ideal generated by $\LL$ is the set of all trace class operators, denoted by $\mathfrak{S}_1$ (notation followed from \cite[Chapter III, Section 7]{GK1969}) which is a soft ideal by Examples (i), but $\mathfrak{S}_{1}$ is not idempotent as $\mathfrak S^2_1 = \mathfrak S_{1/2}$.
\end{remark}

We next prove the main result of this paper that provides an easier criterion to verify the non-simplicity of Lie algebras as compared to the non-softness of $(\LL)$. 

\begin{theorem}\label{theorem}
Let $\LL$ be a Lie algebra of compact operators. If there exists $T \in \LL$ such that the principal $\BH$-ideal generated by $T$, i.e., $(T)$ is not a soft ideal, then $\LL$ is not simple.
\end{theorem}
\begin{proof}
Since $(T)$ is not a soft ideal, $T \not \in \FH$. Otherwise, $(T) = \FH$ which is idempotent and hence a soft ideal. 

We prove the non-simplicity of $\LL$ by explicitly constructing a nontrivial Lie ideal of $\LL$.
If the commutator $[T, S] := TS-ST = 0$ for all $S \in \LL$, then the principal Lie ideal generated by $T$ in $\LL$, denoted by $\left<T\right>$, consists of only scalar multiples of $T$. Since $\LL$ is an infinite-dimensional Lie algebra, $\{0\} \neq \left<T\right> \subsetneq \LL$. So, in this case $\LL$ is not simple.
Suppose there exists $S \in \LL$ such that $[T, S] \neq 0$. Let $A := [T,S]$. Consider the principal Lie ideal $\left<A\right>$ in $\LL$. Since $A \neq 0$, $\left<A\right> \neq 0$. We next prove that $T \not \in \left<A\right>$. Suppose $T \in \left<A\right>$.
Note that every element of $\left<A\right>$ is of the form $\displaystyle{\sum^{n}_{i=1}}A_iTB_i$ for some $n > 1$ where each $A_i$ or $B_i$ is a compact operator and each of these finite sums of operators is in $(T)\KH$. Since $T \in \left<A\right>$, this implies that $T \in (T)\KH$ and hence $(T) \subset (T)\KH$. As $(T)\KH \subset (T)$, one obtains $(T) = (T)\KH$, i.e., $(T)$ is a soft ideal, contradicting the hypothesis that $(T)$ is not soft. Therefore, $T \not \in \left<A\right>$. So the principal Lie ideal $\left<A\right> \neq \LL$, and hence $\LL$ is not simple. 
\end{proof}

A few remarks on Theorem \ref{theorem}:
\begin{remark}\label{soft-thm}
(i) Corollary \ref{C1} provides an efficient way to check the non-softness of a principal $\BH$-ideal. 

(ii) There are infinite-dimensional nonabelian Lie algebras that do not satisfy the hypothesis of Theorem \ref{theorem}. For instance, take $\LL$ to be the symplectic Lie algebra of finite rank operators acting on an infinite-dimensional Hilbert space (see Example \ref{Symplectic} below for the definition). The principal $\BH$-ideal generated by each finite rank operator in $\LL$ is $\FH$ which is a soft ideal. We will prove in Example \ref{Symplectic} that this symplectic Lie algebra is simple. 

(iii) The converse of Theorem \ref{theorem} is not true. For example, consider the Lie algebra of all finite rank operators each of whose matrix representation with respect to a fixed orthonormal basis is upper triangular. Then the set of all nilpotent finite rank operators that are in this Lie algebra is a nontrivial Lie ideal. But the principal $\BH$-ideal generated by every element of this Lie algebra is $\FH$ which is soft.
\end{remark}

In view of Remark \ref{s-number}(iii) which provides an analytic condition to check the non-softness of a principal $\BH$-ideal, one obtains 
\begin{theorem}\label{T2}
Let $\LL$ be a Lie algebra of compact operators such that there exists $T \in \LL$ whose singular number sequence satisfies the $\Delta_2$-condition. Then $\LL$ is not simple.
\end{theorem}

\begin{corollary}\label{KQ}
If $\LL$ is a Lie algebra of compact quasinilpotent operators such that there exists $T \in \LL$ whose singular number sequence satisfies the $\Delta_2$-condition, then $\LL$ is not simple.

In particular, if $\LL$ consists of all compact weighted shift operators, then $\LL$ is not simple.
\end{corollary}
\begin{proof}
The first part of the corollary follows immediately from Theorem \ref{T2}. For the second part, consider the weighted shift operator $T \in\LL$ with weight sequence $\left<1/n\right>$. Then the principal $\BH$-ideal, $(T) = (\diag \left<1/n\right>)$ is not a soft ideal, so by Theorem \ref{theorem}, $\LL$ is not simple. 
\end{proof}

\noindent \textit{On simplicity of Lie algebras of finite rank operators:} If a Lie algebra consists of finite rank operators, then the principal $\BH$-ideal generated by every element of the Lie algebra is a soft ideal, so we cannot invoke Theorem \ref{theorem} to conclude anything on the non-simplicity of the Lie algebra. Nevertheless, an easy case for non-simplicity of such Lie algebras is

\begin{proposition}
Let $\LL$ be a Lie algebra of finite rank operators acting on $\HH$. If there exists $F \in \LL$ such that the trace of $F$ is not zero, then $\LL$ is not simple.
\end{proposition}
\begin{proof}
The derived ideal $[\LL, \LL]$ is the nontrivial Lie ideal in $\LL$.
\end{proof}

We conclude by making a note that the situation seems unpredictable when a Lie algebra consists of only trace-zero finite rank operators as shown in the following examples.

\begin{example}\label{Symplectic}(Symplectic Lie algebra)
For $\HH$ a separable infinite-dimensional complex Hilbert space, consider the symplectic Lie algebra $\LL$ of trace-zero finite rank operators on $\HH \oplus \HH$ as follows. Let $\{e_n\}^{\infty}_{n=1}$ be a fixed orthonormal basis of $\HH$. For $A, B, C, D \in \BH$ such that each of these operators has only finitely many nonzero entries in its matrix representation with respect to the basis $\{e_n\}^{\infty}_{n=1}$, we define 
\begin{equation*}\LL = \left\{\begin{pmatrix}
A & B\\
C & D
\end{pmatrix} \in \mathcal B(\HH \oplus \HH) \mid B^T = B, C^T = -C  \text{ and }  D = -A^T\right\}. \end{equation*}

Observe that $\LL$ is an infinite-dimensional nonabelian Lie algebra that consists of trace-zero finite rank operators. We next prove that $\LL$ is simple. Let $\mathcal J$ be an arbitrary nonzero Lie ideal of $\LL$ and let $T \in \LL$. To show that $\LL$ is simple, we need to prove that $T \in \mathcal J$. Since $T \in \LL$,  
\begin{equation*}T = \begin{pmatrix}
A & B\\
C& -A^T
\end{pmatrix}
\end{equation*} where each block operator has only finitely many nonzero entries in its matrix representation.
For each $N \in \mathbb N$, let 
$\LL_{N}$ be the set of all operators in $\LL$ such that the nonzero entries of the matrix representations of each block operator $A, B ,C$ and $D$ is supported in $N \times N$ square matrix. Choose $N$ large enough such that $T \in \LL_{N}$ and that the Lie ideal $\mathcal J_N := \mathcal J \cap \LL_N \neq \{0\}$ in $\LL_N$. Since $\LL_N$ is a finite-dimensional symplectic Lie algebra, by \cite[Subsection 8.8.3 and Section 8.9]{H03}, $\LL_N$ is simple. Hence $\mathcal J_N = \mathcal L_N$. Moreover, as $T \in \LL_N$ so $T \in \mathcal J_N$ (as $\mathcal J_N = \mathcal L_N$).  And $\mathcal J_N \subset \mathcal J$ implies that $T \in \mathcal J$. Therefore, $\LL$ is simple.
\end{example}

\begin{example}\label{nonsimple}
Consider the Lie algebra $\LL$ that consists of trace-zero finite rank operators that have upper triangular matrix representations with respect to a fixed orthonormal basis of $\HH$. Then $\LL$ is not simple because the Lie ideal of all nilpotent finite rank operators in $\LL$ is a nontrivial Lie ideal of $\LL$. 
\end{example}

These examples motivate the following question.

\begin{question}
Which infinite-dimensional nonabelian Lie algebras of trace-zero finite rank operators are simple?
\end{question}

We remark that if a Lie algebra of trace-zero finite rank operators $\LL$ is strictly triangularizable, then by \cite[Theorem 5.8]{BST10}, $\LL$ is not simple. So, a necessary condition for the simplicity of Lie algebras of finite rank operators is the non-triangularizability of these Lie algebras.

\section*{Acknowledgement}
The author is grateful to Victor Shulman for his invaluable input on the subject.


\section*{References}

\begin{biblist}
\bib{BST10}{article}{
   author={Bre\v{s}ar, Matej},
   author={Shulman, Victor S.},
   author={Turovskii, Yuri V.},
   title={On tractability and ideal problem in non-associative operator
   algebras},
   journal={Integral Equations Operator Theory},
   volume={67},
   date={2010},
   number={2},
   pages={279--300},
   issn={0378-620X},
   review={\MR{2650775}},
   doi={10.1007/s00020-010-1781-z},
}
\bib{C41}{article}{
   author={Calkin, J. W.},
   title={Two-sided ideals and congruences in the ring of bounded operators
   in Hilbert space},
   journal={Ann. of Math. (2)},
   volume={42},
   date={1941},
   pages={839--873},
   issn={0003-486X},
   review={\MR{5790}},
   doi={10.2307/1968771},
}







\bib{D66}{article}{
   author={Dixmier, Jacques},
   title={Existence de traces non normales},
   language={French},
   journal={C. R. Acad. Sci. Paris S\'{e}r. A-B},
   volume={262},
   date={1966},
   pages={A1107--A1108},
   issn={0151-0509},
   review={\MR{196508}},
}

\bib{FR83}{article}{
   author={Fong, C. K.},
   author={Radjavi, H.},
   title={On ideals and Lie ideals of compact operators},
   journal={Math. Ann.},
   volume={262},
   date={1983},
   number={1},
   pages={23--28},
   issn={0025-5831},
   review={\MR{690004}},
   doi={10.1007/BF01474167},
}






\bib{KW02}{article}{
    AUTHOR = {Kaftal, Victor},
    author = {Weiss, Gary},
     TITLE = {Traces, ideals, and arithmetic means},
   JOURNAL = {Proc. Natl. Acad. Sci. USA},
  FJOURNAL = {Proceedings of the National Academy of Sciences of the United
              States of America},
    VOLUME = {99},
      YEAR = {2002},
    NUMBER = {11},
     PAGES = {7356--7360},
      ISSN = {0027-8424},
   MRCLASS = {47L20 (47B47)},
  MRNUMBER = {1907839},
       DOI = {10.1073/pnas.112074699},
       URL = {https://doi.org/10.1073/pnas.112074699},
}
\bib{KW07}{article}{
   author={Kaftal, Victor},
   author={Weiss, Gary},
   title={Soft ideals and arithmetic mean ideals},
   journal={Integral Equations Operator Theory},
   volume={58},
   date={2007},
   number={3},
   pages={363--405},
   issn={0378-620X},
   review={\MR{2320854}},
   doi={10.1007/s00020-007-1498-9},
}

\bib{GK1969}{book}{
    author       = {Israel C. Gohberg}
    author = {Mark Grigor'evich Kre$\breve{\text{i}}$n},
    date         = {1969},
    title        = {Introduction to the theory of linear nonselfadjoint operators},
    doi          = {10.1090/mmono/018},
    isbn         = {978-0-8218-1568-7 (print); 978-1-4704-4436-5 (online)},
    language     = {English},
    origlanguage = {Russian},
    pages        = {xv+378},
    publisher    = {American Mathematical Society, Providence, R.I.},
    series       = {Translations of Mathematical Monographs},
    translator   = {A. Feinstein},
    volume       = {18},
    mrclass      = {47.10},
    mrnumber     = {0246142 (39 \#7447)},
    year         = {1969},
  }
  
 \bib{H03}{book}{
   author={Hall, Brian C.},
   title={Lie groups, Lie algebras, and representations},
   series={Graduate Texts in Mathematics},
   volume={222},
   note={An elementary introduction},
   publisher={Springer-Verlag, New York},
   date={2003},
   pages={xiv+351},
   isbn={0-387-40122-9},
   review={\MR{1997306}},
   doi={10.1007/978-0-387-21554-9},
} 
  
  
  
  
  
  
\bib{M64}{article}{
   author={Mitjagin, B. S.},
   title={Normed ideals of intermediate type},
   language={Russian},
   journal={Izv. Akad. Nauk SSSR Ser. Mat.},
   volume={28},
   date={1964},
   pages={819--832},
   issn={0373-2436},
   review={\MR{0173935}},
}
  
  
  
  
  
  
  
 \bib{P80}{book}{
   author={Pietsch, Albrecht},
   title={Operator ideals},
   series={North-Holland Mathematical Library},
   volume={20},
   note={Translated from German by the author},
   publisher={North-Holland Publishing Co., Amsterdam-New York},
   date={1980},
   pages={451},
   isbn={0-444-85293-X},
   review={\MR{582655}},
} 

\bib{PW12}{article}{
   author={Patnaik, Sasmita},
   author={Weiss, Gary},
   title={Subideals of operators II},
   journal={Integral Equations Operator Theory},
   volume={74},
   date={2012},
   number={4},
   pages={587--600},
   issn={0378-620X},
   review={\MR{3000435}},
   doi={10.1007/s00020-012-2007-3},
}





  
  \bib{PW13}{article}{
   author={Patnaik, Sasmita},
   author={Weiss, Gary},
   title={Subideals of operators},
   journal={J. Operator Theory},
   volume={70},
   date={2013},
   number={2},
   pages={355--373},
   issn={0379-4024},
   review={\MR{3138361}},
   doi={10.7900/jot.2011sep02.1926},
}
  
  
  
  
  

\bib{W76}{article}{
   author={Wojty\'{n}ski, W.},
   title={Banach-Lie algebras of compact operators},
   journal={Studia Math.},
   volume={59},
   date={1976/77},
   number={3},
   pages={263--273},
   issn={0039-3223},
   review={\MR{430806}},
   doi={10.4064/sm-59-3-263-273},
}



\end{biblist}
\end{document}